\documentclass[10pt]{amsproc}
\usepackage{hyperref,color,ytableau,MnSymbol,graphicx}
\newtheorem{theorem}{Theorem}[subsection]
\newtheorem{lemma}[theorem]{Lemma}
\newtheorem{corollary}[theorem]{Corollary}
\theoremstyle{definition}
\newtheorem{definition}[theorem]{Definition}
\theoremstyle{remark}
\newtheorem{remark}[theorem]{Remark}
\newtheorem{example}[theorem]{Example}
\newcommand{\rowins}{\mathrm{RINS}}
\newcommand{\ins}{\mathrm{INSERT}}
\newcommand{\del}{\mathrm{DELETE}}

\newcommand{\wt}{\mathrm{wt}}
\newcommand{\shape}{\mathrm{shape}}
\newcommand{\pl}{\mathrm{pl}}
\newcommand{\ttab}{\mathrm{Tab}^\dagger}
\newcommand{\tr}{\mathrm{R}^\dagger}
\newcommand{\rp}{\mathbf{R}^+}
\newcommand{\rsk}{\mathrm{RSK}}
\newcommand{\ot}{\leftarrow}
\newcommand{\infl}{\mathrm{INFL}}
\newcommand{\supp}{\mathrm{supp}}
\title{A Timed Version of the Plactic Monoid}
\author{Amritanshu Prasad}
\address{The Institute of Mathematical Sciences, Chennai}
\address{Homi Bhabha National Institute, Mumbai}
\thanks{Supported by a Swarna Jayanti Fellowship of the Department of Science \& Technology, Government of India.
  These results were announced at the 83rd annual conference of the Indian Mathematical Society, held at Sri Venkateshwara University, Tirupati where the author delivered the Srinivasa Ramanujan Memorial award lecture.
An expository article on the classical version of Greene's theorem, and its extension to the timed case (based on the lecture) will appear in \emph{The Mathematics Student} \cite{math-student}.}
\begin{document}
\begin{abstract}
  Timed words are words where letters of the alphabet come with time stamps.
  We extend the definitions of semistandard tableaux, insertion, Knuth equivalence, and the plactic monoid to the setting of timed words.
  Using this, Greene's theorem is formulated and proved for timed words, and algorithms for the RSK correspondence are extended to real matrices.
\end{abstract}
\maketitle
\section{Introduction}
Lascoux and Sch\"utzenberger \cite{plaxique} introduced the plactic monoid to give a proof of the Littlewood-Richardson rule based on a strategy outlined by Robinson \cite{robinson-algo}, and ideas of Schensted \cite{schensted} and Knuth \cite{knuth}.
This theory revolves around bijections involving words and tableaux.

These bijections are restrictions to lattice points of certain volume-preserving piecewise linear bijections between convex polyhedra \cite{berkir,kir-trop,DeLoera2004,pak}.
The importance of this viewpoint is borne out in the work of Knutson and Tao, who proved the saturation of the monoid of triples $(\lambda, \mu, \nu)$ of integer partitions such that the representation $V_\lambda$ occurs in the tensor product $V_\mu\otimes V_\nu$ of representations of $GL_n(\mathbf C)$.
This led to the resolution of \emph{Horn's conjecture} on the possible sets of eigenvalues of a sum of Hermitian matrices \cite{knutsontaojams,knutsontaonotices}.

Here we develop the monoid-theoretic foundations for piecewise linear correspondences interpolating bijections involving tableaux.
This is done by generalizing the plactic monoid from the framework of the free monoid of words to the setting of \emph{timed words}.
Timed words were introduced by Alur and Dill \cite{alur-dill} in their approach to the formal verification of real-time systems using timed automata.
While words represent a sequence of events, timed words represent a sequence of events where the time of occurrence of each event is also recorded.
We use a finite version of their definition of timed words.
While each letter occurs discretely (an integer number of times) in a classical word, it appears for a positive duration (which is a real number) in a timed word.

In Section~\ref{sec:timed-tableaux} we introduce the timed versions of semistandard Young tableaux (called \emph{timed tableaux}). 
Schensted's insertion algorithm is generalized to timed tableaux in Section~\ref{sec:timed-insertion}.
Greene invariants of timed words are introduced in Section~\ref{sec:timed-greene-invar}.
Timed versions of Knuth relations are introduced in Section~\ref{sec:timed-knuth-equiv}. 
These relations are conceptually very similar to the relations introduced by Knuth in \cite{knuth}.
However, it was a delicate task to arrive at relations that are at once simple enough so that one can show that they preserve Greene invariants (Lemma~\ref{lemma:Knuth-Greene}) but also powerful enough to transform any timed word to its insertion tableau (Lemma~\ref{lemma:reduction-to-tab}).
With this groundwork, the extension of Greene's theorem to timed words becomes routine (Theorem~\ref{theorem:timed-version-greene}).

Standard properties of Knuth equivalence, such as the existence of a unique tableau in each Knuth class, and the characterization of Knuth equivalence in terms of Greene invariants are extended to timed words in Sections~\ref{sec:tabl-knuth-equiv} and \ref{sec:char-timed-knuth}.

The RSK algorithm \cite{knuth} is extended from integer matrices in Sections~\ref{sec:defin-using-timed} and \ref{sec:insert-record-algor}.
Viennot's light-and-shadows version of the Robinson-Schensted correspondence, which was extended to integer matrices in \cite{rtcv}, is now extended to real matrices.
The piecewise linear nature of these algorithms can be easily seen from the timed version of Greene's theorem.

All the algorithms described here are straightforward to implement.
An implementation in python is available at \url{http://www.imsc.res.in/~amri/timed_plactic/timed_tableau.py}. A jupyter worksheet with demos of many of the theorems and proofs in this paper is available at \url{http://www.imsc.res.in/~amri/timed_plactic/timed_tableau.ipynb} and in html format at \url{http://www.imsc.res.in/~amri/timed_plactic/timed_tableau.html}.
\section{Insertion in Timed Tableaux}
\label{sec:insertion}
\subsection{Timed Tableaux}
\label{sec:timed-tableaux}
Let $A_n=\{1,\dotsc,n\}$, to be thought of as a linearly ordered alphabet.
\begin{definition}
  [Timed Word]
  \label{definition:timed-word}
  A \emph{timed word} of length $r$ in the alphabet $A_n$ is a piecewise-constant right-continuous function $w:[0,r)\to A_n$ with finitely many discontinuities.
  We write $l(w)=r$.
  In other words, for some finite sequence $0=r_0<r_1<\dotsc<r_k=r$ of transition points, and letters $c_1,\dotsc, c_k$ in $A_n$, $w(x) = c_i$ if $r_{i-1}\leq x < r_i$.
  Given such a function, we write
  \begin{equation}
    \label{eq:exp_not}
    w = c_1^{t_1} c_2^{t_2}\dotsb c_k^{t_k},
  \end{equation}
  where $t_i = r_i-r_{i-1}$.
  We call this an \emph{exponential string} for $w$.
  The \emph{weight} of $w$ is the vector:
  \begin{displaymath}
    \wt(w) = (m_1,\dotsc,m_n),
  \end{displaymath}
  where $m_i$ is the Lebesgue measure of the pre-image of $i$ under $w$, in other words,
  \begin{displaymath}
    m_i=\mathrm{meas}(w^{-1}(i)) = \sum_{j\in \{1,\dotsc,k\},\; c_j=i} t_j.
  \end{displaymath}
\end{definition}
The exponential string, as defined above, is not unique; if two successive letters $c_i$ and $c_{i+1}$ are equal, then we can merge them, replacing $c_i^{t_i}c_{i+1}^{t_{i+1}} = c_i^{t_i+t_{i+1}}$.

The above definition is a finite variant of Definition~3.1 of Alur and Dill~\cite{alur-dill}, where $r=\infty$, and there is an infinite increasing sequence of transition points.

Given timed words $w_1$ and $w_2$, their \emph{concatenation} is defined in the most obvious manner---their exponential strings are concatenated (and if necessary, successive equal values merged).
The monoid formed by all timed words in an alphabet $A_n$, with product defined by concatenation, is denoted by $A_n^\dagger$.
The submonoid of $A_n^\dagger$ consisting of timed words where the exponents $t_1,t_2,\dotsc,t_k$ in exponential string (\ref{eq:exp_not}) are integers is the free monoid $A_n^*$ of words in the alphabet $A_n$ (see e.g., \cite[Chapter~1]{comb-words}).
\begin{definition}
  [Timed Subword]
  \label{definition:timed-subword}
  Given a timed word $w:[0,r)\to A_n$, and $S\subset [0,r)$ a finite disjoint union of intervals of the form $[a, b)\subset [0,r)$, the \emph{timed subword $w_S$ of $w$ with respect to $S$} is defined by:
  \begin{displaymath}
    w_S(t) = w(\inf\{u\in [0,r)\mid \mathrm{meas}([0,u)\cap S) \geq t\}) \text{ for } 0\leq t < \mathrm{meas}(S).
  \end{displaymath}
  Intuitively, $w_S$ is obtained from $w$ by cutting out the segments that are outside $S$.
  Given timed words $v$ and $w$, $v$ is said to be a \emph{timed subword} of $w$ if there exists $S\subset [0,r)$ as above such that $v=w_S$.
  Timed subwords $v_1,\dotsc,v_k$ of $w$ are said to be \emph{pairwise disjoint} if there exist pairwise disjoint subsets $S_1,\dotsc,S_k$ as above such that $v_i=w_{S_i}$ for $i=1,\dotsc,k$.
\end{definition}
\begin{definition}[Timed Row]
A \emph{timed row} in the alphabet $A_n$ is a weakly increasing timed word in $A_n^\dagger$.
In exponential notation every timed row is of the form $1^{t_1}\dotsb n^{t_n}$ where $t_i\geq 0$ for $i=1,\dotsc,n$.
The set of all timed rows in $A_n^\dagger$ is denoted $\tr_n$.
The set of all timed rows of length $l$ is denoted $\tr_n(l)$.
\end{definition}
\begin{definition}[Row Decomposition]
Every timed word $w$ has a unique decomposition:
\begin{displaymath}
  w = u_l u_{l-1}\dotsb u_1,
\end{displaymath}
such that $u_i$ is a timed row for each $i=1,\dotsc,l$, and $u_iu_{i-1}$ is not a row for any $i=2,\dotsc,l$.
We shall refer to such a decomposition as the \emph{row decomposition} of $w$.
\end{definition}
Given two timed rows $u$ and $v$, say that $u$ is \emph{dominated} by $v$ (denoted $u\lhd v$) if
\begin{enumerate}
\item $l(u)\geq l(v)$, and
\item $u(t)<v(t)$ for all $0\leq t<l(v)$.
\end{enumerate}
\begin{definition}
  [Real Partition]\label{definition:real-partition}
  A \emph{real partition} is a weakly decreasing finite sequence $\lambda=(\lambda_1,\dotsc,\lambda_l)$ of non-negative real numbers.
  Two real partitions are regarded as equal if one may be obtained from the other by the removal of trailing zeroes.
\end{definition}
\begin{definition}[Timed Tableau]\label{definition:timed-tableau}
  A \emph{timed tableau} in $A_n$ is a timed word $w$ in $A_n$ with row decomposition $w=u_l u_{l-1}\dotsb u_1$ such that $u_1\lhd \dotsb \lhd u_l$.
  The \emph{shape} of $w$ is the real partition $(l(u_1),l(u_2),\dotsc,l(u_l))$, and the \emph{weight} of $w$ is the weight of $w$ as a timed word (see Definition~\ref{definition:timed-word}).
  The set of all timed tableaux in $A_n$ is denoted $\ttab_n$.
  The set of all timed tableaux of shape $\lambda$ is denoted $\ttab_n(\lambda)$.
  The set of all timed tableaux of shape $\lambda$ and weight $\mu$ is denoted $\ttab_n(\lambda,\mu)$.
\end{definition}
The above is a direct generalization of the notion of the reading word of a tableau in the sense of \cite{Lascoux}.
\begin{example}
  \label{example:timed-tableau}
  $w=3^{0.8}4^{1.1}1^{1.4}2^{1.6}3^{0.7}$ is a timed tableau in $A_5$ of shape $(3.7,1.9)$ and weight $(1.4, 1.6, 1.5, 1.1,0)$.
\end{example}
\subsection{Timed Insertion}
\label{sec:timed-insertion}
Given a timed word $w$ and $0\leq a < b \leq l(w)$, according to Definition~\ref{definition:timed-subword}, $w_{[a, b)}$ is the timed word of length $b-a$ such that:
\begin{displaymath}
  w_{[a, b)}(t) = w(a+ t) \text{ for } 0\leq t<b-a.
\end{displaymath}
\begin{definition}[Timed row insertion]
  \label{definition:timed-row-insertion}
  Given a timed row $u$, a letter $c\in A_n$, and a real number $t_c\geq 0$, the insertion $\rowins(u, c^{t_c})$ of $c^{t_c}$ into $u$ is defined as follows:
  if $u(t)\leq c$ for all $0\leq t < l(u)$, then
  \begin{displaymath}
    \rowins(u, c^{t_c}) = (\emptyset, uc^{t_c}),
  \end{displaymath}
  where $\emptyset$ denotes the \emph{empty word} of length zero.
  Otherwise, there exists $0\leq t < l(u)$ such that $u(t)>c$.
  Let
  \begin{displaymath}
    t_0 = \min\{0\leq t< l(u) \mid u(t)> c\}.
  \end{displaymath}
  Define
  \begin{displaymath}
    \rowins(u, c^{t_c}) =
    \begin{cases}
      (u_{[t_0, t_0+t_c)}, u_{[0, t_0)}c^{t_c} u_{[t_0+t_c, l(u))}) & \text{if } l(u) - t_0 > t_c,\\
      (u_{[t_0, l(u))}, u_{[0, t_0)} c^{t_c}) & \text{if } l(u) - t_0 \leq t_c.
    \end{cases}
  \end{displaymath}
  When $t_c=1$ and $u$ lies in the image of $A_n^*$ in $A_n^\dagger$, this coincides with the first step of the algorithm INSERT from Knuth \cite{knuth} where $c$ is inserted into a row:
  if $\rowins(u,c^1)=(v',u')$, then $u'$ is the new row obtained after insertion, $v'$ is the letter bumped out by $c$.

  If $v$ is a row of the form $c_1^{t_1}\dotsb c_k^{t_k}$, define $\rowins(u,v)$ by induction on $k$ as follows:
  Having defined $(v',u')=\rowins(u,c_1^{t_1}\dotsb c_{k-1}^{t_{k-1}})$, let $(v'', u'')=\rowins(u',c_k^{t_k})$.
  Then define
  \begin{displaymath}
    \rowins(u,v) = (v'v'', u'').
  \end{displaymath}
\end{definition}
\begin{remark}
  The insertion of the timed row $v=c_1^{t_1}\dotsb c_k^{t_k}$ into $u$ is achieved by successively inserting $c_i^{t_i}$ as $i$ runs from $1$ to $k$.
  The segments are ejected are taken from $u$ from left to right, with no overlaps.
  It follows that, if $(v',u')=\rowins(u,v)$, then $v'$ is again a timed row.
\end{remark}
\begin{remark}
  Definition~\ref{definition:timed-row-insertion} is in terms of exponential strings, which are not uniquely associated to timed words.
  Therefore, in order that row insertion be well-defined for timed words, it is necessary to check that the result of row insertion of $c^{t_1+t_2}$ is the same as the result of row insertion of $c^{t_1}$ followed by row insertion of $c^{t_2}$.
  This is straightforward.
\end{remark}
\begin{example}
  \label{example:timed-row-ins}
  $\rowins(1^{1.4}2^{1.6}3^{0.7},1^{0.7}2^{0.2})=(2^{0.7}3^{0.2},1^{2.1}2^{1.1}3^{0.5})$.
\end{example}
\begin{definition}
  [Timed Tableau Insertion]
  \label{definition:timed-tableau-insertion}
  Let $w$ be a timed tableau with row decomposition $u_l\dotsc u_1$, and let $v$ be a timed row.
  Then $\ins(w, v)$, the insertion of $v$ into $w$, is defined as follows:
  first $v$ is inserted into $u_1$.
  If $\rowins(u_1,v)=(v_1',u_1')$, then $v_1'$ is inserted into $u_2$; if $\rowins(u_2,v_1')=(v_2',u_2')$, then $v_2'$ is inserted in $u_3$, and so on.
  This process continues, generating $v_1',\dotsc,v_l'$ and $u_1',\dotsc,u_l'$.
  $\ins(w,v)$ is defined to be $v_l'u_l'\dotsb u_1'$.
  It is quite possible that $v_l'=\emptyset$.
\end{definition}
\begin{example}
  If $w$ is the timed tableau from Example~\ref{example:timed-tableau}, then
  \begin{displaymath}
    \ins(w,1^{0.7}2^{0.2})=3^{0.7}4^{0.2}2^{0.7}3^{0.3}4^{0.9}1^{2.1}2^{1.1}3^{0.5}.
  \end{displaymath}
\end{example}
When the timed tableau $w$ lies in the image of $A_n^*$, then $\ins(w,c^1)$ is the same as the result of applying the algorithm $\text{INSERT}(c)$ from Knuth \cite{knuth} to the semistandard Young tableau $w$.
\begin{definition}
  Given real partitions $\lambda=(\lambda_1,\dotsc,\lambda_l)$ and $\mu=(\mu_1,\dotsc,\mu_{l-1})$, we say that $\mu$ \emph{interleaves} $\lambda$ if the inequalities
  \begin{displaymath}
    \lambda_1 \geq \mu_1 \geq \lambda_2 \geq \mu_2 \geq \dotsb \geq \lambda_{l-1}\geq \mu_{l-1}\geq \lambda_l 
  \end{displaymath}
  hold.
  In other words, the successive parts of $\mu$ lie in-between the successive parts of $\lambda$.
\end{definition}
\begin{theorem}
  \label{theorem:tableauness-of-insertion}
  For any timed tableau $w$ in $A_n$ and any timed row $v$ in $A_n$, $\ins(w,v)$ is again a timed tableau in $A_n$.
  We have
  \begin{displaymath}
    \wt(\ins(w,v)) = \wt(w) + \wt(v),
  \end{displaymath}
  and $\shape(w)$ interleaves $\shape(\ins(w,v))$.
\end{theorem}
\begin{proof}
  The main step in this proof is the following lemma:
  \begin{lemma}
    \label{lemma:tableauness-of-insertion}
    Suppose $x$ and $y$ are timed rows in $A_n$ such that $x\lhd y$.
    For any timed row $v$ in $A_n$, suppose $(v',x')=\rowins(x,v)$, and $(v'',y')=\rowins(y,v')$.
    \begin{enumerate}
    \item \label{item:dom} $x'\lhd y'$
    \item \label{item:wt} $\wt(v'')+\wt(x')+\wt(y')=\wt(x)+\wt(y)+\wt(v)$
    \item \label{item:int} $l(y')\leq l(x)$.
    \end{enumerate}
  \end{lemma}
  \begin{proof}[Proof of the lemma]
    By inserting in stages, assume that $v=c^t$ for some $c\in A_n$ and some $t>0$.

    If $x$ never exceeds $c$, then $x'=xc^t$, and $y'=y$, and so $x'\lhd y'$.
    Otherwise, when $c^t$ is inserted into $x$, $v'$ is a segment of $x$ corresponding to an interval $[t_0,t_0+\delta)$ such that $x(t_0)>c$.
    This segment in $x$ is replaced by a segment $c^\delta$ to obtain $x'$.
    Let $c_1^{t_1}\dotsb c_k^{t_k}$, with $c<c_1< \dotsb <c_k$, be the exponential string of $v'$.

    Proceed by induction on $k$.
    If $k=1$, $v'=c_1^{t_1}$.
    Now $y(t_0)>x(t_0)=c_1$, so $c_1^{t_1}$ will displace a segment of $y$, one that begins to the left of $t_0$, with $c_1^{t_1}$, and so $x'\lhd y'$.

    For $k>1$, perform the insertion of $c^t$ into $x$ in two steps, first inserting $c^{t_1}$, and then inserting $c^{t-t_1}$.
    If $(v_1',x_1')=\rowins(x,c^{t_1})$, then $v_1'=c_1^{t_1}$.
    Let $(v_1'',y_1')=\rowins(y,v_1')$.
    By the $k=1$ case, $x_1'\lhd y_1'$.
    
    Now $(c_2^{t_2}\dotsb c_k^{t_k},y')=\rowins(y_1',c^{t-t_1})$, so $y'$ is obtained by inserting $c_2^{t_2}\dotsb c_k^{t_k}$ into $y_1'$.
    Therefore, by induction hypothesis, $x'\lhd y'$, proving (\ref{item:dom}).

    The assertion (\ref{item:wt}) about weights is straightforward.
    For (\ref{item:int}), observe that $v'$ is a concatenation of segments from $x$.
    Write $v'=wz$, where $w$ consists of segments that come from $u_{[0,l(y))}$ and $z$ consists of segments that come from $u_{[l(y),l(x))}$.
    Then, from the arguments in the proof of the first part of the lemma, the segments in $w$ will all replace segments of $y$, so if $(v'',y'')=\rowins(y,w)$, then $l(y'')=l(y)$.
    Now $(v',y')=\rowins(y'',z)$, whence $l(y')\leq l(y'')+l(z)\leq l(y)+[l(x)-l(y)] = l(x)$.
  \end{proof}
  We now prove that $\ins(w,v)$ is a timed tableau.
  Suppose $w$ has row decomposition $u_lu_{l-1}\dotsb u_1$.
  Using the notation of Definition~\ref{definition:timed-tableau-insertion}, and writing $v_0$ for $v$, we have $(v_i',u'_i)=\rowins(u_i,v'_{i-1})$ and $(v_{i+1}',u_{i+1}')=\rowins(u_{i+1},v'_i)$.
  The first assertion of Lemma~\ref{lemma:tableauness-of-insertion}, with $v=v'_i$, $x=u_i$, and $y=u_{i+1}$, shows that $u'_i\lhd u'_{i+1}$ for $i=1,\dotsc,l-1$.
  Taking $x=u_l$ and $y=\emptyset$ gives $u'_l\lhd v'_l$. Therefore $\ins(w,v)$ is a timed tableau.
  The third assertion of Lemma~\ref{lemma:tableauness-of-insertion}, with the same settings, gives $l(u_{i+1}')\leq l(u_i)$, showing that $\shape(w)$ interleaves $\shape(\ins(w,v)$.
  The assertion about weights in the theorem follows easily from the second assertion of Lemma~\ref{lemma:tableauness-of-insertion}.
\end{proof}
\begin{definition}
  [Insertion Tableau of a Timed Word]
  \label{definition:insertion_tableaux}
  Let $w$ be a timed word with row decomposition $u_1\dotsb u_l$.
  The insertion tableau of $w$ is defined as:
  \begin{displaymath}
    P(w) = \ins(\dotsb\ins(\ins(u_1, u_2),u_3),\dotsc,u_l).
  \end{displaymath}
\end{definition}
\begin{example}
  If $w=3^{0.8}1^{0.5}4^{1.1}1^{0.9}2^{1.6}3^{0.7}1^{0.7}2^{0.2}$ has four rows in its row decomposition.
  $P(w)$ is calculated via the following steps:
  \begin{displaymath}
    \renewcommand{\arraystretch}{1.2}
    \begin{array}{|l|l|}
      \hline
      w & P(w)\\
      \hline
      3^{0.8} & 3^{0.8}\\
%      3^{0.8}1^{0.5} & 3^{0.5}1^{0.5}3^{0.3}\\
      3^{0.8}1^{0.5}4^{1.1} & 3^{0.5}1^{0.5}3^{0.3}4^{1.1}\\
%      3^{0.8}1^{0.5}4^{1.1}1^{0.9} & 3^{0.8}4^{0.6}1^{1.4}4^{0.5}\\
%      3^{0.8}1^{0.5}4^{1.1}1^{0.9}2^{1.6} & 3^{0.8}4^{1.1}1^{1.4}2^{1.6}\\
      3^{0.8}1^{0.5}4^{1.1}1^{0.9}2^{1.6}3^{0.7} & 3^{0.8}4^{1.1}1^{1.4}2^{1.6}3^{0.7}\\
%      3^{0.8}1^{0.5}4^{1.1}1^{0.9}2^{1.6}3^{0.7}1^{0.7} & 3^{0.7}2^{0.7}3^{0.1}4^{1.1}1^{2.1}2^{0.9}3^{0.7}\\
      3^{0.8}1^{0.5}4^{1.1}1^{0.9}2^{1.6}3^{0.7}1^{0.7}2^{0.2} & 3^{0.7}4^{0.2}2^{0.7}3^{0.3}4^{0.9}1^{2.1}2^{1.1}3^{0.5}\\
      \hline
    \end{array}
    \renewcommand{\arraystretch}{1}
  \end{displaymath}
\end{example}
\begin{definition}
  [Sch\"utzenberger Involution on Timed Words]
  \label{definition:schuetzenberger-involution}
  Given $w=c_1^{t_1}\dotsb c_k^{t_k}\in A_n^\dagger$, define
  \begin{equation}
    \label{eq:sharp}
    w^\sharp = (n-c_k+1)^{t_k} \dotsb (n-c_1+1)^{t_1},
  \end{equation}
  in effect, reversing both the order on the alphabet, and the positional order of letters in the timed word.
\end{definition}
\begin{lemma}
  \label{lemma:reverse-row-insertion}
  Let $u$ and $v$ be timed rows.
  Suppose $\rowins(u,v)=(v',u')$, and $l(v')=l(v)$.
  Then $\rowins({u'}^\sharp,{v'}^\sharp)=(v^\sharp,u^\sharp)$.
\end{lemma}
\begin{proof}
  It suffices to consider the case where $v=c^t$.
  The hypothesis $l(v')=l(v)$ implies that $t_0=\inf\{t\mid u(t)>c\}$ satisfies $0\leq t_0\leq l(u)-c$, and
  \begin{displaymath}
    u'=u_{[0,t_0)}c^tu_{[t_0+t,l(u))}, \text{ and } v'=u_{[t_0,t_0+t)}.
  \end{displaymath}
  Using induction as in the proof of Theorem~\ref{theorem:tableauness-of-insertion}, we may assume that $v'$ is constant, so $v'=d^t$ for some $d>c$.

  Now
  \begin{displaymath}
    {u'}^\sharp=u_{[t_0+t,l(u))}^\sharp (n-c+1)^t u_{[0,t_0)}^\sharp \text{ and } {v'}^\sharp=(n-d+1)^t.
  \end{displaymath}
  Since all the values of $u_{[t_0+t,l(u))}$ are greater than or equal to $d$, all the values of $u_{[t_0+t,l(u))}^\sharp$ are less than or equal to $n-d+1$.
  Moreover, $n-c+1>n-d+1$.
  It follows immediately from Definition~\ref{definition:timed-row-insertion} that $\rowins({u'}^\sharp,{v'}^\sharp)=(v^\sharp,u^\sharp)$.
\end{proof}
\begin{corollary}
  \label{corollary:row-insertion-bijection}
  The timed row insertion algorithm gives rise to a bijection:
  \begin{multline*}
    \rowins: \tr_n(r)\times \tr_n(s) \tilde\to \\\{(v',u')\in \tr_n(r+s-r')\times \tr_n(r')\mid r'\geq \max(r,s),\; u'\lhd v'\}. 
  \end{multline*}
\end{corollary}
\begin{proof}
  Suppose $(u,v)\in \tr_n(r)\times \tr_n(s)$, and $(v',u')=\rowins(u,v)$.
  Then $(u,v)$ can be recovered from $(v',u')$ (given the prior knowledge of $r$ and $s$) as follows:
  let $(v_1^\sharp, u_1^\sharp)=\rowins({u'_{[0,r)}}^\sharp, {v'}^\sharp)$.
  Then using Lemma~\ref{lemma:reverse-row-insertion}, $u$ and $v$ can be recovered as $u=u_1$, and $v=v_1u'_{[r,r')}$.
\end{proof}
\begin{theorem}[Timed Pieri Rule]
  \label{theorem:pieri}
  The timed insertion algorithm gives rise to a bijection:
  \begin{displaymath}
    \ins: \ttab_n(\lambda)\times \tr_n(r) \tilde\to \coprod_{\begin{smallmatrix}\text{$\lambda$ interleaves $\mu$}\\{l(\lambda)+r = l(\mu)}\end{smallmatrix}} \ttab_n(\mu)
  \end{displaymath}
\end{theorem}
\begin{proof}
  Let $\lambda=(\lambda_1,\dotsc,\lambda_l)$.
  Let $w\in \ttab_n(\lambda)$ have row decomposition $u_l\dotsb u_1$, and $x\in \tr_n(r)$.
  Suppose that $w'=\ins(w,x)$ has row decomposition $u'_{l+1}\dotsb u'_1$ (with the possibility that $u'_{l+1}=\emptyset$).
  We already know that $\shape(w)$ interleaves $\shape(w)$ (Theorem~\ref{theorem:tableauness-of-insertion}).
  Given timed rows $u'$ and $v'$ such that $u'\lhd v'$, and non-negative real numbers $r$ and $s$ such that $r\leq l(u')$, let $\rowins^{-1}_r(v',u')$ denote the unique pair of rows $(u,v)$ such that $l(u)=r$, $l(v)=s$, and $(v',u')=\rowins(u,v)$ (see Corollary~\ref{corollary:row-insertion-bijection}).
  Then the rows of $w$ can be recovered from $w'$ as follows:
  \begin{align*}
    (x_l, u_l) & = \rowins^{-1}_{\lambda_l}(u'_{l+1},u'_l),\\
    (x_{l-1},u_{l-1}) & = \rowins^{-1}_{\lambda_{l-1}}(x_l,u'_{l-1}),\\
    (x_{l-2},u_{l-2}) & = \rowins^{-1}_{\lambda_{l-2}}(x_{l-1},u'_{l-2}),\\
    &\vdots\\
    (x_1,u_1) & = \rowins^{-1}_{\lambda_1}(x_2,u'_1),
  \end{align*}
  and finally $x$ can be recovered as $x=x_1$.
\end{proof}
\begin{definition}
  [Deletion]
  \label{definition:deletion}
  Let $w'\in \ttab_n(\mu)$ and let $\lambda$ be a real partition that interleaves $\mu$.
  Then we write
  \begin{displaymath}
    \del_\lambda(w')=(v,w) \text{ if and only if } w\in \ttab_n(\lambda) \text{ and }\ins(w,v)=w'.
  \end{displaymath}
  The pair $(v,w)$ is computed from $w'$ and $\lambda$ by the algorithm described in the proof of Theorem~\ref{theorem:pieri}.
\end{definition}
\section{Greene's Theorem}
\label{sec:greene}
\subsection{Greene Invariants for Timed Words}
\label{sec:timed-greene-invar}
\begin{definition}[Greene Invariants for Timed Words]
  \label{definition:timed-Greene-invars}
  Given $w\in A_n^\dagger$, its \emph{$k$th Greene invariant} $a_k(w)$ is defined as the maximum possible sum of lengths of a set of $k$ pairwise disjoint subwords of $w$ (see Definition~\ref{definition:timed-subword}) that are all timed rows:
  \begin{multline*}
    a_k(w) = \sup\{l(u_1)+\dotsb+l(u_k)\mid u_1,\dotsc,u_k \text{ are pairwise disjoint subwords,}\\ \text{and each $u_i$ is a timed row}\}
  \end{multline*}
\end{definition}
\begin{lemma}
  \label{lemma:tableau-shape-greene}
  If $w$ is a timed tableau of shape $\lambda=(\lambda_1,\dotsc,\lambda_l)$, then for each $1\leq k\leq l$,
  \begin{displaymath}
    a_k(w) = \lambda_1+\dotsb + \lambda_k.
  \end{displaymath}
\end{lemma}
\begin{proof}
  This proof is very similar to the proof of the corresponding result for ordinary tableaux \cite{Greene-schen,Lascoux}.
  Suppose $w$ has row decomposition $u_lu_{l-1}\dotsb u_1$.
  Then $u_1,\dotsc,u_k$ are pairwise disjoint subwords that are rows, so
  \begin{displaymath}
    a_k(w) \geq \lambda_1+\dotsb + \lambda_k.
  \end{displaymath}
  Conversely, any row subword of $w$ cannot have overlapping segments from two different rows $u_i$ and $u_j$ of $w$, because if $i>j$, then $u_i(t)>u_j(t)$, but in the row decomposition of $w$, $u_i$ occurs before $u_j$.
  Therefore, $k$ disjoint subwords can have length at most the sum of lengths of the largest $k$ rows of $w$, which is $\lambda_1+\dotsc+\lambda_k$.
\end{proof}
\subsection{Timed Knuth Equivalence and the Timed Plactic Monoid}
\begin{definition}
  [Timed Knuth Relations]
  \label{sec:timed-knuth-equiv}
  Assume that $x$, $y$ and $z$ are timed rows such that $xyz$ is also a timed row.
  The timed Knuth relations are given by:
  \begin{align}
    \tag{$\kappa_1$}
    \label{eq:tk1}
    xzy & \equiv zxy \text{ if } l(z)=l(y) \text{ and } \lim_{t\to l(y)^-} y(t)<z(0),\\
    \tag{$\kappa_2$}
    \label{eq:tk2}
    yxz & \equiv yzx\text{ if } l(x)=l(y) \text{ and } \lim_{t\to l(x)^-} x(t)<y(0).
  \end{align}
\end{definition}
\begin{definition}
  [Timed Plactic Monoid]
  \label{definition:timed-plactic-monoid}
  The \emph{timed plactic monoid} $\pl^\dagger(A_n)$ is the quotient $A_n^\dagger/\equiv$, where $\equiv$ is the congruence generated by the timed Knuth relations (\ref{eq:tk1}) and (\ref{eq:tk2}).
\end{definition}
In other words, two elements of $A_n^\dagger$ are said to \emph{differ by a Knuth relation} if they are of the form $uv_1w$ and $uv_2w$, where $v_1$ and $v_2$ are terms on opposite sides of one of the timed Knuth relations (\ref{eq:tk1}) and (\ref{eq:tk2}).
Knuth equivalence $\equiv$ is the equivalence relation generated by Knuth relations.
Since this equivalence is stable under left and right multiplication in $A_n^\dagger$, the concatenation product on $A_n^\dagger$ descends to a product on the set $\pl^\dagger(A_n)$ of Knuth equivalence classes, giving it the structure of a monoid.
\begin{lemma}
  \label{lemma:sharp-moves}
  Two timed words $v$ and $w$ differ by a Knuth relation (\ref{eq:tk1}) if and only if $v^\sharp$ and $w^\sharp$ (see Definition~\ref{definition:schuetzenberger-involution}) differ by a Knuth relation (\ref{eq:tk2}).
\end{lemma}
\begin{proof}
  When the involution $w\mapsto w^\sharp$ is applied to the Knuth relation (\ref{eq:tk1}), the Knuth relation (\ref{eq:tk2}) is obtained.
\end{proof}
\begin{lemma}
  \label{lemma:reduction-to-tab}
  Every timed word is Knuth equivalent to its timed insertion tableau.
\end{lemma}
\begin{proof}
  First we show that, for timed rows $u$ and $v$, if $(v',u')=\rowins(v,u)$, then $v'u'\equiv uv$.
  By insertion in stages, we assume that $v=c^t$.
  If $u(t)\leq c$ for all $0\leq t<l(u)$, there is nothing to show.
  Otherwise, a segment $y$ of $u$, beginning at $t_0$, and of length $t_1=\min(l(u)-t_0,t)$ is displaced by the segment $c^{t_1}$ of $c^t$.
  Write $u=x'yx''$.
  It suffices to show $x'yx''c^{t_1}\equiv yx'c^{t_1}x''$.
  But this can be done in two steps as follows (the segment to which the Knuth relation is applied is underlined):
  \begin{displaymath}
    x'\underline{y x'' c^{t_1}} \equiv_{\kappa_2} x'\underline{yc^{t_1}x''} = \underline{x'yc^{t_1}}x'' \equiv_{\kappa_1} \underline{yx'c^{t_1}}x''.
  \end{displaymath}
  From Definition~\ref{definition:insertion_tableaux}, it suffices to show that $\ins(w,v)\equiv wv$ for every timed tableau $w$ and every timed row $v$.
  Suppose $w$ has row decomposition $u_lu_{l-1}\dotsb u_1$ then, with the notations of Definition~\ref{definition:timed-tableau-insertion},
  \begin{align*}
    wv &= u_l\dotsb u_2u_1 v\\
       &\equiv u_l \dotsb u_2v_1'u_1'\\
       &\equiv u_l \dotsb v_2'u_2'u_1'\\
       & \vdots\\
       &\equiv v_l'u_l'\dotsb u_2'u_1' = \ins((w,v),
  \end{align*}
  by repeated application of the assertion at the beginning of this proof.
\end{proof}
\subsection{Knuth Equivalence and Greene Invariants}
\label{sec:knuth-equiv-green}
\begin{lemma}
  \label{lemma:Knuth-Greene}
  If two timed words are Knuth equivalent, then they have the same Greene invariants.
\end{lemma}
\begin{proof}
  It suffices to prove that if two words differ by a Knuth relation they have the same Greene invariants.
  For the Knuth relation (\ref{eq:tk1}), suppose that $xyz$ is a timed row with $l(z)=l(y)$, and the last letter of $y$ is strictly less than the first letter of $z$.
  For any timed words $w$ and $u$, we wish to show that Greene invariants coincide for $wxzyu$ and $wzxyu$.
  Since every timed row subword of $wzxyu$ is also a timed row subword of $wxzyu$, $a_k(wzxyu)\leq a_k(wxzyu)$ for all $k$.
  
  To prove the reverse inequality, for any set of pairwise disjoint row subwords $v_1,\dotsc,v_k$ of $wxzyu$, it suffices to construct pairwise disjoint row subwords $v_1',\dotsc,v_k'$ of $wzxyu$ such that $\sum_{i=1}^k l(v'_k)\geq \sum_{i=1}^k l(v_k)$.
  Write $v_i=w_ix_iz_iy_iu_i$ for each $i$, where $w_i,x_i,z_i,y_i$ and $u_i$ are (possibly empty) row subwords of $w,x,z,y$ and $u$ respectively.

  Since the last letter of $y$ is strictly smaller than the first letter of $z$, it cannot be that $y_i\neq \emptyset$ and $z_i\neq \emptyset$ simultaneously for the same $i$.
  If, for all $i$, $x_i=\emptyset$, or $z_i=\emptyset$, then each $v_i$ remains a row subword of $wzxyu$, so we may take $v'_i=v_i$ for all $i$.

  Otherwise, there exists $i$ such that $v_i=w_ix_iz_iu_i$, with $x_i\neq \emptyset$, and $z_i\neq \emptyset$.
  Without loss of generality, assume that this is the case for $i=1,\dotsc,r$, and not for $i=r+1,\dotsc,k$ for some $1\leq r\leq k$.
  If $y_i=\emptyset$ for all $i$, then set
  \begin{displaymath}
    v'_i =
    \begin{cases}
      w_1x_1yu_1&\text{for }i=1,\\
      w_ix_iu_i&\text{for }1<i\leq r,\\
      v_i&\text{for }r<i\leq k.
    \end{cases}
  \end{displaymath}
  Since $l(y)=l(z)\geq \sum_{i=1}^r l(z_i)$, $\sum_{i=1}^k l(v'_i)\geq \sum_{i=1}^k l(v_i)$.
  By construction the words $v'_1,\dotsc,v'_k$ are pairwise disjoint timed row subwords of $w$.

  Finally, suppose there exists at least one index $i$ such that $y_i\neq \emptyset$, say $y_{r+1}\neq \emptyset$.
  Also, assume that $\max z_1$ (the largest letter of $z_1$) is at least as large as $\max z_i$ for $i=2,\dotsc,r$.
  Let $z_0$ be the timed row obtained by concatenating all the segments of $z_i$ for $1\leq i\leq r$ in the order in which they occur in $z$.
  It follows that $\max z_0 = \max z_1$, and $l(z_0)=l(z_1)+\dotsb+l(z_r)$.
  Let $x_0$ be the timed row obtained by concatenating all the segments of $x_1$ and $x_{r+1}$ in the order in which they occur in $x$.
  If $\min x_1\leq \min x_{r+1}$, define
  \begin{displaymath}
    v'_i = 
    \begin{cases}
      w_1x_0y_{r+1}u_{r+1} & \text{for }i=1,\\
      w_ix_iu_i & \text{for }1<i\leq r,\\
      w_{r+1}z_0u_1 &\text{for }i=r+1,\\
      v_i&\text{for }r+1<i\leq k.
    \end{cases}
  \end{displaymath}
  If $\min x_1>\min x_{r+1}$, then interchange $w_1$ and $w_{r+1}$ in the above definition.
  In both cases the words $v'_1,\dotsc, v'_k$ are pairwise disjoint row subwords of $wzxyu$ whose lengths add up to $l$.

  For the Knuth relation (\ref{eq:tk2}), a similar argument can be given.
  However, a more elegant method is to use Lemma~\ref{lemma:sharp-moves}, noting that $a_k(w)=a_k(w^\sharp)$ for all $k\geq 1$ and all $w\in A_n^\dagger$, thereby reducing it to (\ref{eq:tk1}).
\end{proof}\pagebreak
\subsection{The timed version of Greene' theorem}
\label{sec:timed-version-greene}
\begin{theorem}
  [Timed version of Greene's theorem]
  \label{theorem:timed-version-greene}
  For every $w\in A_n^\dagger$, if the timed tableau $P(w)$ has shape $\lambda=(\lambda_1,\dotsc,\lambda_l)$, then
  \begin{displaymath}
    a_k(w) = \lambda_1+\dotsb+\lambda_k \text{ for $k=1,\dotsc,l$}.
  \end{displaymath}
\end{theorem}
\begin{proof}
  Greene's theorem holds when $w$ is a timed tableau (Lemma~\ref{lemma:tableau-shape-greene}).
  By Lemma~\ref{lemma:Knuth-Greene}, Greene invariants remain unchanged under the timed versions of Knuth relations.
  By Lemma~\ref{lemma:reduction-to-tab}, every timed word is Knuth equivalent to its timed insertion tableau.
  Therefore, the Greene invariants of a timed word are given by the shape of its insertion tableau as stated in the theorem.
\end{proof}
\begin{remark}
  The proof of Lemma~\ref{lemma:tableau-shape-greene} shows that the supremum in the definition of Greene invariants (Definition~\ref{definition:timed-Greene-invars}) is attained for timed tableaux.
  From the proof of Lemma~\ref{lemma:Knuth-Greene}, it follows that this supremum is attained for every timed word in the Knuth equivalence class of a timed tableau, and therefore for every $w\in A_n^\dagger$.
\end{remark}
\section{Knuth Equivalence Classes}
\label{sec:knuth-classes}
\subsection{Tableaux in Knuth Equivalence Classes}
\label{sec:tabl-knuth-equiv}
Given $w\in A_n^\dagger$, let $\bar w$ denote the word in $A_{n-1}^\dagger$ whose exponential string is obtained by removing all terms of the form $n^t$ with $t>0$ from the exponential string of $w$.
The word $\bar w$ is called the restriction of $w$ to $A_{n-1}$.
\begin{lemma}
  \label{lemma:restriction-interleaf}
  For every timed tableau $w\in A_n^\dagger$, $\bar w$ is also a timed tableau. Moreover, $\shape(\bar w)$ interleaves $\shape(w)$.
\end{lemma}
\begin{proof}
  Suppose $w$ has row decomposition $u_lu_{l-1}\dotsb u_1$.
  Since $n$ is the largest element of $A_n$, we may write $u_i=u'_in^{t_i}$ for some $t_i\geq 0$.
  Clearly $l(u_i)\geq l(u'_i)$.
  Since $w$ is semistandard, $l(u'_i)\geq l(u_{i+1})$ for $i=1,\dotsc,l-1$. 
  It follows that the shape of $w'$, which is $(l(u'_1),\dotsc,l(u'_l))$ interleaves the shape of $w$, which is $(l(u_1),\dotsc,l(u_l))$.
  Since $u_i\lhd u_{i+1}$, it follows that $u'_i\lhd u_{i+1}$ for $i=1,\dotsc,l-1$.
\end{proof}
\begin{lemma}
  \label{lemma:equivalence-restriction}
  If $v,w\in A_n^\dagger$ are Knuth equivalent, then their restrictions to $A_{n-1}$, $\bar v$ and $\bar w$ are Knuth equivalent in $A_{n-1}^\dagger$.
\end{lemma}
\begin{proof}
  Applying the restriction to $A_{n-1}$ map $w\mapsto \bar w$ to both sides of the Knuth relation (\ref{eq:tk1}) gives:
  $x\bar z y$ and $\bar z x y$.
  Write $y=y'y''$, where $l(y')=l(\bar z)$, we have
  \begin{displaymath}
    x\bar zy'y'' \equiv \bar zxy'y'',
  \end{displaymath}
  a Knuth relation in in $A_{n-1}^\dagger$.
  A similar argument works for the Knuth relation (\ref{eq:tk2}).
\end{proof}
\begin{theorem}
  \label{theorem:unique-timed-tableaux}
  Every Knuth equivalence class in $A_n^\dagger$ contains a unique timed tableau.
\end{theorem}
\begin{proof}
  The existence of a timed tableau in each Knuth equivalence class is ensured by Lemma~\ref{lemma:reduction-to-tab}.
  The proof of uniqueness is by induction on $n$.
  The base case, where $n=1$ is trivially true.
  Now suppose $v$ and $w$ are Knuth equivalent timed tableaux in $A_n^\dagger$.
  By Lemmas~\ref{lemma:restriction-interleaf} and~\ref{lemma:equivalence-restriction} $\bar v$ and $\bar w$ are Knuth equivalent timed tableaux in $A_{n-1}^\dagger$.
  By the induction hypothesis, $\bar v=\bar w$.
  Let $\lambda=(\lambda_1,\dotsc,\lambda_l)$ be the shape of this timed tableau.
  By Lemma~\ref{sec:knuth-equiv-green}, $v$ and $w$ have the same Greene invariants, and therefore the same shape $\mu=(\mu_1,\dotsc,\mu_{l+1})$.
  It follows that both $v$ and $w$ are obtained from $\bar v=\bar w$ by appending $n^{\mu_i-\lambda_i}$ to the $i$th row of $\bar v= \bar w$ for each $i$, hence $v=w$.
\end{proof}
\subsection{Characterization of Knuth Equivalence Classes}
\label{sec:char-timed-knuth}
Classical Knuth equivalence can be characterized in terms of Greene invariants (see \cite[Theorem~2.15]{plaxique}).
The same characterization works for timed Knuth equivalence.
\begin{theorem}
  Let $w$ and $w'$ be timed words in $A_n^\dagger$.
  Then $w$ and $w'$ are Knuth equivalent if and only if, for all timed words $u$ and $v$ in $A_n^\dagger$, $a_k(uwv)=a_k(uw'v)$ for all $k\geq 1$.
\end{theorem}
\begin{proof}
  If $w$ and $w'$ are Knuth equivalent, then so are $uwv$ and $uw'v$.
  By the timed version of Greene's theorem (Theorem~\ref{theorem:timed-version-greene}) $a_k(uwv)=a_k(uw'v)$ for all $u,v\in A_n^\dagger$.

  For the converse, suppose that $w$ and $w'$ are not Knuth equivalent.
  Then $P(w)\neq P(w')$.
  If $P(w)$ and $P(w')$ do not have the same shape, then by Theorem~\ref{theorem:timed-version-greene}, they do not have the same Greene invariants, so taking $u=v=\emptyset$ proves the result.

  Now suppose that $w$ and $w'$ are rows of the same length.
  If $w\neq w'$, there exist decompositions $w=xc^ty$ and $w'=x{c'}^ty'$, where $c\neq c'$, and $t>0$.
  If $c<c'$, then for $T>t+l(y)$,
  \begin{displaymath}
    a_1(xc^tyc^T) = l(x)+t+T, \text{ while } a_1(x{c'}^tyc^T) = l(x)+T,
  \end{displaymath}
  thereby proving the result.

  In the general case, suppose $w=u_lu_{l-1}\dotsb u_1$ and $w'=u'_lu'_{l-1}\dotsb u'_1$ are row decompositions.
  Let $i$ be the least integer such that $u_i\neq u'_i$.
  By the proof for rows, there exists $c\in A_n$, and $T>0$ such that when $(v,x)=\rowins(u_i,c^T)$ and $(v',x')=\rowins(u'_i,c^T)$, then $l(x)\neq l(x')$.
  Also, note that $c$ is at least $i$, the least possible value of the $i$th row of a tableau.

  Now assume that $T>l(u_j)$ for $j=1,\dotsc,i-1$.
  Take $u=1^T2^T\dotsb (i-1)^T$.
  Then we have
  \begin{align*}
    P(uw)&=u_l\dotsb u_{i+1} i^T u_i (i-1)^T u_{i-1} \dotsb 1^T u_1,\\
    P(uw')&=u'_l\dotsb u'_{i+1} i^T u'_i (i-1)^T u_{i-1} \dotsb 1^T u_1.\\
  \end{align*}
  Now $z=c^T(i-1)^T u_{i-1} \dotsb 1^T u_1$ is a timed tableau.
  Let 
  \begin{displaymath}
    (v,\bar z)=\del_{(T+\lambda_1,\dotsc,T+\lambda_{i-1})}(z).
  \end{displaymath}
  Then when $P(uwv)$ and $P(uw'v)$ are computed, the calculations are the same for the first $i-1$ rows.
  But then $c^T$ is inserted into $u_i$ and $u'_i$ to obtain the $i$th rows of $P(uwv)$ and $P(uw'v)$, which, by our earlier argument, will have different lengths.
\end{proof}
\section{The Real RSK Correspondence}
\label{sec:rsk}
\subsection{Definition using Timed Insertion Tableaux}
\label{sec:defin-using-timed}
Let $M_{m\times n}(\rp)$ denote the set of all $m\times n$ matrices with non-negative real entries.
Given $A=(a_{ij})\in M_{m\times n}(\rp)$, define its \emph{timed column word} $u_A$, and \emph{timed row word} $v_A$ as follows:
\begin{align*}
  u_A & = 1^{a_{11}}2^{a_{12}}\dotsb n^{a_{1n}}\,1^{a_{21}}2^{a_{22}}\dotsb n^{a_{2n}}\,\dotsb \,1^{a_{m1}}2^{a_{m2}}\dotsb n^{a_{mn}}.\\
  v_A & = 1^{a_{11}}2^{a_{21}}\dotsb m^{a_{m1}}\,1^{a_{12}}2^{a_{22}}\dotsb m^{a_{m2}}\,\dotsb \,1^{a_{1n}}2^{a_{2n}}\dotsb m^{a_{mn}}.
\end{align*}
The timed word $u_A$ is obtained by reading column numbers of $A$ along its rows, timed by its entries.
The timed word $v_A$ is obtained by reading the row numbers of $A$ along its columns, timed by its entries.
Define:
\begin{equation}
  \label{eq:rsk}
  \rsk(A) = (P(u_A), P(v_A)).
\end{equation}
This is a direct generalization of the definition of the RSK correspondence given in \cite[Section~18]{schur_poly}.
\begin{example}
  Let
  \begin{displaymath}
    A = 
    \begin{pmatrix}
      0.16 & 0.29 & 0.68 & 0.44\\ 
      0.29 & 0.70 & 0.38 & 0.45\\ 
      0.32 & 0.29 & 0.43 & 0.70
    \end{pmatrix}.
  \end{displaymath}
  Then $(P, Q) = \rsk(A)$ are given by:
  \begin{align*}
    P & = 3^{0.32}4^{0.29}2^{0.60}3^{0.65}4^{0.55}1^{0.77}2^{0.67}3^{0.52}4^{0.75},\\
    Q & = 3^{0.61}2^{1.38}3^{0.43}1^{1.57}2^{0.45}3^{0.70},
  \end{align*}
  which have common shape $(2.71,1.81,0.61)$.
  Visually, four different colours can be used to depict the letters of our alphabet:
  \begin{center}
    \includegraphics[width=0.7\textwidth]{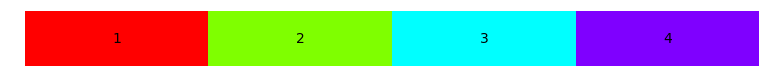},
  \end{center}
  then the tableaux $P$ and $Q$ can be represented by the images:
  \begin{center}
    \includegraphics[height=2cm]{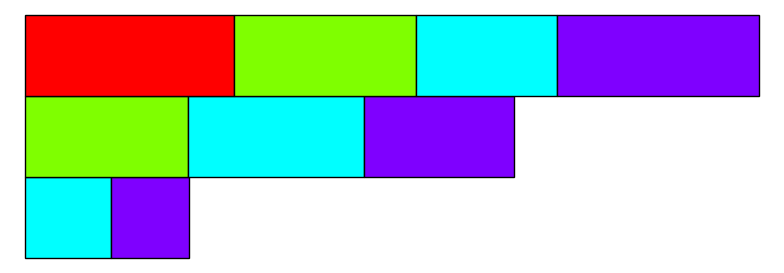}\hspace{1cm}\includegraphics[height=2cm]{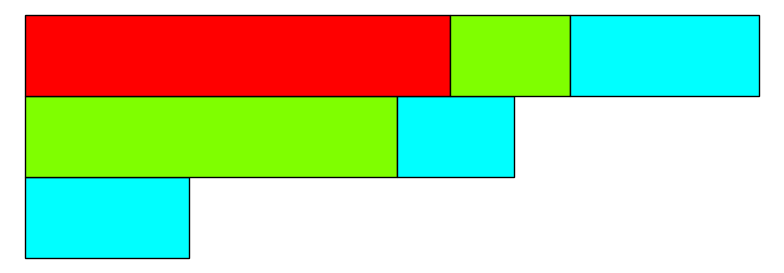}
  \end{center}
  From this visual representation, it is evident that $P$ and $Q$ are timed tableaux of the same shape.
\end{example}
\begin{theorem}
  \label{theorem:rsk}
  The function $\rsk$ defines a bijection:
  \begin{displaymath}
    \rsk: M_{m\times n}(\rp)\tilde\to \coprod_\lambda \ttab_n(\lambda)\times\ttab_m(\lambda),
  \end{displaymath}
  where $\lambda$ runs over all real partitions with at most $\min(m,n)$ parts.
\end{theorem}
\begin{remark}
  Let $\mu_i$ denote the sum of the $i$th row of $A$, and $\nu_j$ the sum of the $j$th column.
  Let $\mu=(\mu_1,\dotsc,\mu_m)$, and $\nu=(\nu_1,\dotsc,\nu_n)$.
  Then, if $\rsk(A)=(P,Q)$ then $\wt(P)=\nu$, and $\wt(Q)=\mu$.
\end{remark}
\begin{remark}[Relation to Knuth's definition]
  Knuth~\cite{knuth} defined $\rsk(A)=(P,Q)$ for integer matrices in a slightly different manner.
  His definition of $P=P(u_A)$ is exactly the same as the definition here.
  However $Q$ is defined as a \emph{recording tableau} which has the same shape as $P$ by its very construction.
  With Knuth's construction, each step (insertion followed by recording) is reversible, and it is clear that a bijection is obtained.
  The symmetry property of the RSK correspondence, that $\rsk(A^T)=(Q,P)$ if $\rsk(A)=(P,Q)$ is then stated as a non-trivial theorem.

  The definition (\ref{eq:rsk}) is the extension to real matrices of the definition in \cite[Section~18]{schur_poly} for integer matrices.
  With this definition it is immediate that if $\rsk(A)=(P,Q)$, then $\rsk(A^T)=(Q,P)$ since $u_{A^T}=v_A$.
  However, it is not immediately clear that $P$ and $Q$ have the same shape, and that the correspondence is invertible.
  These are proved using Greene's theorem in \cite{schur_poly}.
  For real matrices, the timed version of Greene's theorem allows the proof of \cite{schur_poly} to be carried out for real matrices.
  For the sake of completeness, this argument is given in full detail below.
\end{remark}
\begin{lemma}
  \label{lemma:same-shape}
  For every $A\in M_{m\times n}(\rp)$, the tableaux $P(u_A)$ and $P(v_A)$ have the same shape.
\end{lemma}
\begin{proof}
  Any timed subword $w$ of $u_A$ is of the form
  \begin{displaymath}
    w=1^{b_{11}}2^{b_{12}}\dotsb n^{b_{1n}}\,1^{b_{21}}2^{b_{22}}\dotsb n^{b_{2n}}\,\dotsb \,1^{b_{m1}}2^{b_{m2}}\dotsb n^{b_{mn}},    
  \end{displaymath}
  where $0\leq b_{ij}\leq a_{ij}$ for all $(i,j)$.
  If $w$ is a row, then the indices $(i_1,j_1),\dotsc,(i_k,j_k)$ for which $b_{ij}>0$, when taken in the order in which they appear in $w$, must satisfy $i_1\leq \dotsc \leq i_k$, and $j_1\leq \dotsb \leq j_k$.
  Define a partial order on the set
  \begin{displaymath}
    P_{mn} = \{(i,j)\mid 1\leq i\leq m,1\leq j \leq n\}
  \end{displaymath}
  by $(i,j)\leq (i',j')$ if and only if $i\leq i'$ and $j\leq j'$.
  Then it follows that the $k$th timed Greene invariant of $u_A$ (Definition~\ref{definition:timed-Greene-invars}) is given by:
  \begin{displaymath}
    a_k(u_A) = \max_C \sum_{(i,j)\in C} a_{ij},
  \end{displaymath}
  where the maximum is taken over the set of all subsets $C\subset P_{mn}$ which can be written as a union of $k$ chains.
  Since the order relation on $P_{mn}$ corresponds to the order relation on $P_{nm}$ under $(i,j)\leftrightarrow (j,i)$, it follows that $a_k(v_A)=a_k(u_A)$ for all $k$.
  Thus, the timed version of Greene's theorem (Theorem~\ref{sec:timed-version-greene}) implies that $P$ and $Q$ have the same shape.
\end{proof}
\subsection{Insertion-Recording Algorithm for $\rsk(A)$}
\label{sec:insert-record-algor}
Given real partitions $\lambda$ and $\mu$ such that $\lambda$ interleaves $\mu$, and $w\in \ttab_{m-1}(\lambda)$, define the \emph{inflation of $w$ to shape $\mu$ by $m$} to be the unique tableau $\infl_\mu(w,m)$ of shape $\mu$ whose restriction to $m-1$ is equal to $w$.
In the notation of Section~\ref{sec:tabl-knuth-equiv}, $\overline{\infl_\mu(w,m)}=w$. 

Given $A\in M_n(\rp)$, let $r_{i,A} = 1^{a_{i1}}2^{a_{i2}}\dotsb n^{a_{in}}$.
Then $u_A=r_{1,A}r_{2,A}\dotsb r_{m,A}$.
\begin{center}
  \textbf{Insertion-Recording Algorithm}
\end{center}
\begin{itemize}
\item $P\ot \emptyset$, $Q\ot \emptyset$.
\item For $i=1,\dotsc, m$, repeat the following steps:
  \begin{itemize}
  \item $P\ot \ins(P, r_{i,A})$.
  \item $\lambda \ot \shape(P)$.
  \item $Q\ot \infl_\lambda(Q,i)$
  \end{itemize}
\item Return $(P, Q)$.
\end{itemize}
\begin{lemma}
  \label{lemma:insertion-rec-algo}
  For every $A\in M_{m\times n}(\rp)$, the output of the insertion-recording algorithm is $\rsk(A)$ as defined in (\ref{eq:rsk}).
\end{lemma}
\begin{proof}
  The proof is by induction on the number $m$ of rows in $A$.
  The base case of $m=1$ is trivial.

  Now suppose $A'$ denotes the submatrix consisting of the first $m-1$ rows of $A$.
  Then $u_A = u_{A'}r_{m,A}$, so that $P(u_A)=\ins(P(u_{A'}),r_{m,A})$,
  Also, the restriction $\bar v_A$ of $v_A$ to $m-1$ is $v_{A'}$.
  
  Since $v_{A'}$ is the restriction of $v_A$ to $A_{m-1}$, by Lemma~\ref{lemma:equivalence-restriction}, $P(v_{A'})$ is Knuth equivalent to the restriction of $P(v_A)$ to $A_{m-1}$.
  Theorem~\ref{theorem:unique-timed-tableaux}, implies that $P(v_{A'})$ is equal to the restriction of $P(v_A)$ to $A_{m-1}$.
  Therefore $P(v_A)=\infl_\lambda(P(v_{A'}), m)$, which is the output of the insertion-recording algorithm.
\end{proof}
\begin{proof}
  [Proof of Theorem~\ref{theorem:rsk}]
  The proof uses the fact that the insertion-recording algorithm is invertible.
  Following the notation of the proof of Lemma~\ref{lemma:insertion-rec-algo}, it suffices to recover $r_{m,A}$, $P(u_{A'})$ and $P(v_{A'})$ from $P(u_A)$ and $P(v_A)$ to reverse the insertion-recording algorithm.
  For this, observe that $P(v_{A'})$ is just the restriction of $P(v_A)$ to $A_{m-1}$.
  If $\mu$ is the shape of $P(v_{A'})$, then $(r_{m,A},P(u_{A'}))=\del_\mu(P(u_A))$ (see Definition~\ref{definition:deletion}).
\end{proof}
\subsection{Light-and-Shadows Real RSK}
\label{sec:light-and-shadows-rsk}
Viennot described a visual version of the Robinson-Schensted-Correspondence for permutations, using the \emph{light and shadows} method \cite{viennot1977forme}.
This algorithm was extended to the RSK correspondence on integer matrices by Fulton using the matrix-ball method \cite{fulton}.
Another such extension, called the VRSK algorithm, was given in \cite[Chapter~3]{rtcv}.
In VRSK one can work directly with the matrices themselves, without having to draw them as configurations of points in the plane, which get unwieldy when matrices have large entries.
Another unforeseen advantage of the VRSK algorithm is that a minor variant works for real matrices, giving the correspondence of (\ref{eq:rsk}).
This new algorithm, which we call the \emph{light-and-shadows real RSK} is introduced in this section.
The piecewise linear nature of the RSK correspondence becomes clear from this algorithm.
\begin{definition}
  [Sequence of Leading Points]
  For a matrix $A\in M_{m\times n}(\rp)$, consider the set
  \begin{displaymath}
    \supp(A) = \{(i,j)\in P_{mn}\mid a_{ij}>0\}.
  \end{displaymath}
  Then the sequence $L(A)$ of leading points of $A$ is the set $\max(\supp(A))$ (with respect to the poset structure on $P_{mn}$) arranged in a sequence
  \begin{displaymath}
    L(A) = (i_1,j_1),\dotsc,(i_r,j_r)
  \end{displaymath}
  such that $j_1<\dotsb <j_r$ and (since this set is an antichain in $P_{mn}$) $i_1>\dotsb >i_r$.
\end{definition}
\begin{center}
  \textbf{Light-and-Shadows Real RSK}
\end{center}
\begin{itemize}
\item $P\ot \emptyset$, $Q\ot\emptyset$.
  \textcolor{blue}{
  \item While $A$ is non-zero repeat the following steps:
    \begin{itemize}
    \item Set $S\ot 0_{m\times n}$ ($m\times n$ zero matrix)
    \item Set $p=\emptyset$, $q=\emptyset$.
    \item While $A$ is non-zero repeat the following steps:
      \begin{itemize}
      \item Compute $L(A) = (i_1,j_1),\dotsc,(i_r,j_r)$ of $A$
      \item Let $m(A)=\min\{a_{i_1j_1},\dotsc,a_{i_r,j_r}\}$
      \item Set $a_{i_s,j_s}\ot a_{i_s,j_s}-m(A)$ for $s=1,\dotsc r$
      \item Set $s_{i_{s+1},j_s}\to s_{i_{s+1},j_s}+m$ for $s=1,\dotsc,r-1$
      \item Set $p\ot pj_1^m$, $q\ot qi_r^m$
      \end{itemize}
    \item $P\ot pP$, $Q\ot qQ$
    \item $A\ot S$
    \end{itemize}
  }
\item Return $(P, Q)$
\end{itemize}
\begin{theorem}
  When the light-and-shadows real RSK algorithm is applied to $A\in M_{m\times n}(\rp)$, it return $\rsk(A)$.
\end{theorem}
\begin{proof}
  For the proof, we introduce an algorithm that is midway between the insertion-recording algorithm of Section~\ref{sec:insert-record-algor} and the light-and-shadows real RSK.
  \begin{center}
    \textbf{Row-wise RSK Algorithm}
  \end{center}
  \begin{itemize}
  \item $P\ot \emptyset$, $Q\ot \emptyset$
    \textcolor{blue}{
    \item While $A$ is non-zero repeat the following steps:
      \begin{itemize}
      \item Set $p\ot \emptyset$, $q\ot \emptyset$
      \item For $i=1,\dotsc,m$, repeat the following steps:
        \begin{itemize}
        \item Set $S\ot 0_{m\times n}$.
        \item Set $(v,u)=\rowins(p,1^{a_{i1}}\dotsb n^{a_{in}})$
        \item If $v=1^{s_1}\dotsb n^{s_n}$, set $s_{ij}\ot s_j$ for $j=1,\dotsc,n$.
        \end{itemize}
      \item Set $A\ot S$.
      \item $P\ot pP$, $Q\ot \infl_{\shape(P)}(Q,i)$.
      \end{itemize}
      }
    \item return $(P, Q)$
  \end{itemize}
  The main loop of this algorithm starts with a matrix $A$, and replaces it with the matrix $S=s_{ij}$ computed using timed row insertion.
  It also computes the first row of the tableau $P$ and $Q$ as $p, q$.
  We claim that the function $A\mapsto (S, p, q)$ of the main loop of the row-wise RSK algorithm is the same as the function $A\mapsto (S, p, q)$ of the main loop of the light-and-shadows real RSK.
  We call $S$ the \emph{shadow matrix} of $A$.

  Write $A$ as a block matrix $\binom{A'}{A''}$, where $A'$ is an $(m-1)\times n$ matrix and $A''$ is a $1\times n$ matrix.
  It suffices to show that if the light-and-shadows real RSK algorithm return $(P',Q')$ on $A'$ and $(P, Q)$ on $A$, the $P=\ins(P',u_{A''})$.
  Here $u_{A''}$ is just the row:
  \begin{displaymath}
    1^{a_{m1}}\dotsb n^{a_{mn}}.
  \end{displaymath}

  The inner loop of the light-and-shadows real RSK algorithm produces a sequence $A'=A'_1, A'_2,\dotsc, A'_h$ of matrices as it runs on input $A'$.
  Let $L'_k=L(A'_k)$ and $m'_k=m(A'_k)$, for $k=1,\dotsc,h$.
  When the inner loop finishes running, we have $p'={j'_1}^{m'_1}\dotsb {j'_h}^{m'_h}$, and $q' = {i'_1}^{m'_1}\dotsb {i'_h}^{m'_h}$, where $j'_k$ is the least non-zero column, and $i'_k$ is the least non-zero row of $A'_k$.

  Now $A$ is obtained from $A'$ by adding a new row
  $
  \begin{pmatrix}
    a_{m1}&\dotsb & a_{mn}
  \end{pmatrix}
  $.
  To begin with, assume that this row has only one non-zero entry, $a_{mj_0}$.
  Let $L_1,L_2,\dotsc$, and $m_1,m_2,\dotsc$ be the corresponding sequences of leading points, and their corresponding smallest entries respectively.
  If $j_0\geq j_i$ for all $i$, then $L_k=L'_k$ for all $k=1,\dotsc,h$.
  In addition, $A$ has a singleton sequence of leading points $\{(m,j_0)\}$.
  As a result, the output of the main loop is $p=p'u_{A''}$, and $q=q'm^{a_{mj_0}}$, and the shadow matrix of $A$ is the same as the shadow matrix of $A'$.
  The same outcome is obtained from the main loop of the row-wise RSK algorithm.
  The hypothesis that $j_0\geq j_i$ for all $i$ is equivalent to saying that $A''$ has its non-zero entries to the left of any non-zero entries of $A'$.
  Therefore $P(u_A)=P(u_{A'})u_{A''}$, and the shadow matrix of $A'$ generated by row-wise RSK is the same as the shadow matrix of $A$ generated by row-wise RSK.
  
  Now suppose that $j_0<j_l$ for some $l$, and take the least such value $l$.
  Then the sequences of leading points $L_1,\dotsc,L_{l-1}$ of $A$ are the same as the sequences $L'_1,\dotsc,L'_{l-1}$.
  If $a_{mj_0}\geq m'_{j'_l}+\dots m'_{j'_h}$, then $L_k=\{(m,j_0)\}\cup L'_k$ for $k=l,\dotsc, h$.
  Therefore $p={j'_1}^{m'_1}\dotsb {j'_{l-1}}^{m'_{l-1}}j_0^{a_{mj_0}}$.
  From the definition of timed row insertion (Definition~\ref{definition:timed-row-insertion}), $\rowins(p',j_0^{a_{mj_0}}) = (j_l^{m_l}\dotsb j_h^{m_h},p)$.
  Also, $q=q'm^{a_{mj_0}}$.
  Finally, $S$ is obtained from $S'$ by adding $m'_k$ to the $(m,j_k)$th entry of $S'$ for each $k=l,\dotsc,h$.

  Otherwise, $m_{j_l}+\dotsb+m_{j_{q-1}}<a_{mj_0}\leq m_{j_1}+\dotsb+m_{j_q}$ for some $l\leq q<h$.
  In this case the sequences of leading points for $A$ are given by:
  \begin{displaymath}
    L_k =
    \begin{cases}
      L'_k & \text{for } 1\leq k<l-1,\\
      (m,j_0)\cup L'_k & \text{for } l\leq k\leq q,\\
      L'_{k-1} & \text{for } q\leq l\leq h,
    \end{cases}
  \end{displaymath}
  $p={j'_1}^{m'_1}\dotsb j_{l-1}^{m'_{l-1}}j_0^{a_{m0}}j_q^{m'_{j_1}+\dotsb + m'_{j_q}-a_{mj_0}}j_{q+1}^{m'_{q+1}}\dotsb j_h^{m'_h}$.
  Again from the definition of timed row insertion, $\rowins(p',j_0^{a_{mj_0}}) = (j_l^{m'_l}\dotsb j_{q-1}^{m'_{q-1}}j_q^{a_{mj_0}-(m'_1+\dotsb+m'_{l-1})}, p)$.
  Also, $q=q'm^{a_{mj_0}}$.
  Finally, the value $m'_k$ is added to the $(m,j_k)$th entry of $S'$ for $k=l,\dotsc,q-1$, and $a_{mj_0}-(m'_1+\dotsb+m'_{l-1})$ is added to the $(m,j_q)$th entry of $S'$ to obtain $S$.

  Thus we have seen that when $A''$ has a single non-zero entry, the effect of this entry modifies the outputs of both the row-wise real RSK algorithm and the light-and-shadows real RSK algorithm in exactly the same manner.

  If the last row of $A$ has more than one non-zero entry, they may be dealt with sequentially (from left to right) to get the same outcome.
\end{proof}
\subsection{Piecewise Linear RSK}
\label{sec:piecewise-linear-rsk}
\begin{definition}
  [Gelfand-Tsetlin Pattern]
  A \emph{Gelfand-Tsetlin pattern} of size $n$ is a triangle $T=(\lambda^{(k)}_i\mid 1\leq k \leq n, 1\leq i \leq k)$ of non-negative real numbers:
  \begin{displaymath}
    \begin{matrix}
      \lambda^{(n)}_1 & & \lambda^{(n)}_2 && \dotsb && \dotsb && \lambda^{(n)}_n\\
      &\lambda^{(n-1)}_1 && \lambda^{(n-1)}_2 && \dotsb && \lambda^{(n-1)}_{n-1} & \\
      && \ddots && \ddots && \udots &&  \\
      &&& \ddots && \udots &&&\\
      &&&& \lambda^{(1)}_1 &&&&
    \end{matrix}
  \end{displaymath}
  such that $\lambda^{(k)}_i\geq \lambda^{(k-1)}_i\geq \lambda^{(k)}_{i+1}$ for $k=2,\dotsc,n$ and $i=1,\dotsc, k-1$.
  The shape of a Gelfand-Tsetlin pattern of size $n$ is its \emph{top row}, $\lambda^{(n)}$.
\end{definition}
Whenever $n\geq k$, define $r^n_k:A_n^\dagger\to A_k^\dagger$ by taking $r^n_k(w)$ to be the timed word whose exponential string is obtained from the exponential string of $w$ by deleting all terms of the form $c^t$ where $c>k$.
It is easy to see that if $w\in A_n^\dagger$ is a timed tableau, then so is $r^n_k(w)$ for all $k=1,\dotsc,n$.

In this section, given a timed tableau $w$ in $A_n$, its shape will always be written as a real partition with $n$ components, $\lambda=(\lambda_1,\dotsc,\lambda_n)$, where $\lambda_1\geq \dotsb \lambda_n\geq 0$.

Given a timed tableau $w$ in $A_n$, define partitions $\lambda^{(k)}=(\lambda^{(k)}_1,\dotsc,\lambda^{(k)}_k)$ by $\lambda^{(k)} = \shape(r^n_k(w))$.
By Lemma~\ref{lemma:restriction-interleaf}, $\lambda_{(k-1)}$ interleaves $\lambda^{(k)}$ for all $k=2,\dotsc,n$.
Therefore the numbers $\lambda^{(k)}_i$ form a Gelfand-Tsetlin pattern, which we denote by $GT(w)$.
The shape of $\lambda$ is also the shape of $GT(w)$.
Conversely, given a Gelfand-Tsetlin pattern $T$, it is easy to reconstruct the unique timed tableau $w$ such that $T=GT(w)$.

The space of all Gelfand-Tsetlin patterns of size $n$, being defined by a finite collection of homogeneous inequalities in $\binom{n+1}2$ variables, forms a polyhedral cone in $R^{\binom{n+1}2}$.
In terms of the notation introduced in Section~\ref{sec:defin-using-timed}, we have the following lemma, which is well-known for the RSK correspondence on integer matrices (see e.g., \cite[Prop.~2.26]{kir-trop}).
\begin{lemma}
  \label{lemma:pl}
  Given $A\in M_{m\times n}(\rp)$ with $\rsk(A)=(P,Q)$, let $(\lambda^{(k)}_i)=GT(P)$ and $(\mu^{(k)}_i)=GT(Q)$, Gelfand-Tsetlin patterns of size $n$ and $m$ respectively.
  Then
  \begin{align}
    \label{eq:gt-1}
    \lambda^{(j)}_1 + \dotsb \lambda_k^{(j)} & = \max_{C\subset P_{mj}\text{ union of at most $k$ chains}} \sum_{(i,j)\in C} a_{ij},\\
    \label{eq:gt-2}
    \mu_1^{(i)} + \dotsb + \mu_k^{(i)} & = \max_{C\subset P_{in}\text{ union of at most $k$ chains}} \sum_{(i,j)\in C} a_{ij},
  \end{align}
  for all $1\leq j\leq n$, and $1\leq i\leq m$.
\end{lemma}
\begin{proof}
  Let $A_j$ is the submatrix of $A$ consisting of its first $j$ columns, then $u_{A_j} = r^n_j(u_A)$.
  By a repeated application of Lemma~\ref{lemma:equivalence-restriction}, $P(u_{A_j}) = r^n_j(P(u_A))$.
  But the $j$ row of $GT(P)$ is, by definition, the shape of $r^n_j(P(u_A))$.
  But the shape of $P(u_{A_j})$ is given by (\ref{eq:gt-1}), as explained in the proof of Lemma~\ref{lemma:same-shape}.
  The identity (\ref{eq:gt-2}) has a similar proof.
\end{proof}
\begin{corollary}
  The RSK correspondence defines a continuous piecewise linear bijection from the cone $M_{m\times n}(\rp)$ onto the cone of pairs of Gelfand-Tsetlin patterns $((\lambda^{(k)}_i), (\mu^{(k)}_i))$ of sizes $n$ and $m$ respectively, with $\lambda^{(n)}=\mu^{(m)}$ (after padding the shorter of the two with zeros).
\end{corollary}
Lemma~\ref{lemma:pl} clearly demonstrates the piecewise linear nature of the RSK correspondence:
The algorithms of Sections~\ref{sec:defin-using-timed}--\ref{sec:light-and-shadows-rsk} allow for fast computations of this piecewise linear map.
While Eqs.~(\ref{eq:gt-1}) and (\ref{eq:gt-2}) are used to \emph{define} the RSK correspondence in \cite{kir-trop}, in this article, the RSK \emph{algorithm} is extended to real matrices, and Lemma~\ref{lemma:pl} is proved for the extended algorithm.

A more detailed analysis of the piecewise linear nature of the real RSK correspondence will be carried out in \cite{cgp}.
\subsection*{Acknowledgements}
I thank R. Ramanujam for introducing me to the work of Alur and Dill on timed automata.
I thank Xavier Viennot for teaching me about the wonderful world of the RSK correspondence, and encouraging me to work on this project.
I thank Maria Cueto, Sachin Gautam, Sridhar P. Narayanan, Digjoy Paul, K.~N.~Raghavan, Ramanathan Thinniyam, and S.~Viswanath for many interesting discussions, which helped me develop the ideas in this article.
I thank Darij Grinberg for carefully reading, and sending me corrections to an earlier version of this article.
\bibliographystyle{abbrv}
\bibliography{refs}
\end{document}